\documentclass[aop,MSNbibl,number,citesort,dvips]{arximspdf}
\usepackage{mathbh}

\doi{10.1214/11-AOP716}
\volume{41}
\issue{3B}
\pubyear{2013}
\firstpage{2376}
\lastpage{2400}

\makeatletter
\newcommand{\q}[1]{\langle #1\rangle}
\newcommand{\smallset}[1]{\{#1\}}
\let\phi\varphi

\let\sign=\sgn
\def\esssup{\operatorname{esssup}}
\def\eps{\varepsilon}
\def\Isymb{{\mathbh1}}
\newcommand{\I}[1]{\Isymb_{(#1)}}
\def\real{{\mathbb R}}
\def\card{{\mathbb N}}
\def\F{{\mathcal F}}
\def\G{{\mathcal G}}
\def\A{{\mathcal A}}
\let\cl\overline
\newcommand{\zfrac}[2]{\biggl(\frac{#1}{#2}\biggr)}
\newcommand{\PQ}{\mathbf{Q}}
\def\med{\operatorname{med}}
\def\eqinlaw{\stackrel{d}{=}}

\def\z{\mathfrak{z}}
\def\bF{\bar{\F}}
\def\bG{\bar{\G}}

\def\bTV{\bar{\TV}}
\def\bB{\bar{B}}
\def\bxi{\bar{\xi}}
\def\bE{\bar{E}}
\def\bS{\bar{S}}
\def\bsigma{\bar{\sigma}}
\def\bU{\bar{U}}
\def\bW{\bar{W}}
\def\bL{\bar{L}}
\newcommand{\TV}{\mathcal{V}} 
\newcommand{\cC}{\mathcal{C}} 
\newcommand{\tgo}{_{t \geq0}}
\newcommand{\eqref}[1]{(\ref{#1})}
\newtheorem{theorem}{Theorem}
\newtheorem{corollary}[theorem]{Corollary}
\newproclaim{remark}[theorem]{Remark}
\newtheorem{lemma}[theorem]{Lemma}
\newtheorem{proposition}[theorem]{Proposition}
\newproclaim{notation}{Notation}
\newcommand{\fraca}[2]{{#1}/{#2}}
\makeatother

\begin{document}
\begin{frontmatter}

\title{The solution of the perturbed Tanaka-equation is pathwise
unique\thanksref{TT1}}
\runtitle{Perturbed Tanaka-equation}

\begin{aug}
\author[A]{\fnms{Vilmos} \snm{Prokaj}\corref{}\ead[label=e1]{prokaj@cs.elte.hu}}
\runauthor{V. Prokaj}
\affiliation{E\"otv\"os Lor\'and University}
\address[A]{Department of Probability Theory \\ \quad and Statistics\\
E\"otv\"os Lor\'and University\\
1117 Budapest, P\'azm\'any P. s\'et\'any 1/C\\
Hungary\\
\printead{e1}} 
\end{aug}
\thankstext{TT1}{Supported by the European Union and co-financed
by the European Social Fund (grant agreement no. TAMOP
4.2.1./B-09/1/KMR-2010-0003).}

\received{\smonth{4} \syear{2011}}
\revised{\smonth{8} \syear{2011}}

%
\begin{abstract}
The Tanaka equation $dX_t=\sign(X_t)\,dB_t$ is an example of a stochastic
differential equation (SDE)
without strong solution. Hence pathwise uniqueness does not hold for this
equation. In this note we prove that if we modify the right-hand side
of the equation, roughly speaking, with a strong enough additive
noise, independent of the Brownian motion $B$,
then the solution of the
obtained equation is
pathwise unique.
\end{abstract}

%

\end{frontmatter}

\section{Introduction}\label{sec1}
Let $(\Omega,(\F_t)\tgo,\mathbf P)$ be a filtered probability space and
$B=(B^{(1)},B^{(2)})$ be a two-dimensional Brownian motion
in the filtration
$(\F_t)\tgo$.
In the simplest form we are interested in the uniqueness of the
solution for
the following equation:
%
%
%
\begin{equation}
\label{eqXU}
dX_t=\sign(X_t)\,dB^{(1)}_t+\lambda \,dB^{(2)}_t,
\end{equation}
where $\lambda\in\real$ is a constant, and
$\sign$ denotes the signum function taking $-1$ at zero, that is,
$\sign(x)=\I{x>0}-\I{x\leq0}$.
We call \eqref{eqXU} the perturbed Tanaka equation, and the statement
in title
reads as follows:
%
%
\begin{theorem}\label{thm1}
For $\lambda\neq0$ the solution of \eqref{eqXU}
is pathwise unique.
\end{theorem}

Actually we prove a more general statement than Theorem~\ref{thm1}.
For the sake of fluent
composition, we use the term \textit{strongly orthogonal
} for continuous local martingales
whose product is a local martingale, that is, for $M,N$ if $\q
{M,N}=0$. We say
that \textit{$N$ dominates $M$} if for some constant $c>0$ we have
$d\q
{M}\leq
cd\q{N}$. In\vadjust{\goodbreak} other words there is a process $Q$ (it can be chosen to be
predictable) such that $\q{M}_t=\int_0^t Q_s \,d\q{N}_s$ for all
$t\geq0$
and $\mathbf{P}(\forall s\geq0, 0\leq Q_s\leq c)=1$. A localized
version of this
notion, namely \textit{$N$ locally dominates $M$}, holds if this $Q$
is locally
bounded.
%
%
\begin{theorem}\label{thm2}
Let $M,N$ be continuous local martingales in $(\F_t)\tgo$. Assume
that $M$
and $N$ are strongly orthogonal 
and $N$ dominates $M$. Then, the solution of the
equation
%
%
%
\begin{equation}
\label{eqXMN}
dX_t=\sign(X_t)\,dM_t+dN_t
\end{equation}
is pathwise unique.
\end{theorem}

The interest in the uniqueness of the solution of this type of equation stems
from the search for the strong solution of the drift hiding problem. Weak
solution was given in~\cite{drift2009}, and the results of this paper
make it
possible to modify the construction to obtain a strong solution. It is
presented in the forthcoming paper~\cite{drift2009a}. It uses Theorem
\ref{thm2} as a main new ingredient.
Besides this particular application, we think that this problem is also
interesting in its own right.

By standard localization argument, we obtain the following:
%
%
\begin{corollary}\label{cor3}
Let $M,N$ be continuous local martingales in $(\F_t)\tgo$. Assume
that $M$
and $N$ are strongly orthogonal, 
and $N$ locally dominates~$M$. Then, the
solution of \eqref{eqXMN} is pathwise unique.
\end{corollary}

Another trivial extension is obtained by a measure change argument.
%
%
\begin{corollary}\label{cor4}
Let $M,N$ be continuous semimartingales in $(\F_t)\tgo$. Assume that for
each $T\geq0$ there is an equivalent probability measure $\PQ$ on $\F
_T$ such
that $(M_t)_{t\in[0,T]}$
and $(N_t)_{t\in[0,T]}$ are strongly orthogonal local martingales
under $\PQ$,
and $N$ locally dominates $M$. Then, the
solution of \eqref{eqXMN} is pathwise unique.
\end{corollary}

For the proof of pathwise uniqueness, one usually considers $X-X'$
where $X,X'$ are two processes satisfying the equation with the same driving
semimartingale and starting from the same initial value. Here it is not enough;
we also have to deal with $X+X'$. The next theorem essentially states the
uniqueness in terms of $U=(X-X')/2$ and $V=(X+X')/2$.
%
%
\begin{theorem}\label{thm3}
Assume that $U,V$ are
continuous, strongly orthogonal
local martingales
such that
%
%
%
\begin{equation}\label{eqUV}
dU_t=\I{|V_t|<|U_t|}\,dU_t,\qquad U_0=V_0=0.
\end{equation}
If $V$ dominates $U$, then $U$ is trivial, that is, identically zero.
\end{theorem}

Without domination the statement is not true in general. In
Section~\ref{secex} below, we construct a pair $(U,V)$
satisfying\vadjust{\goodbreak}
\eqref{eqUV} such that $U$ is nontrivial. 
By Remark~\ref{remequi} below, this example also
shows that strong orthogonality together with the almost sure absolute
continuity of $\q{M}$ with respect to $\q{N}$ is not enough in Theorem
\ref{thm2} and Corollary~\ref{cor3}. Hence the assumption that $N$ dominates
$M$ is essential. Moreover, it is possible to construct an example in
which $M$ is a Brownian motion, and the perturbation $N$ is such that its
quadratic variation is equivalent with the Lebesgue measure almost
surely, and
still the pathwise uniqueness does not hold for \eqref{eqXMN}.
Even if the perturbation $N$ is a Brownian
motion, one can construct a local martingale $M$ strongly orthogonal to $N$
such that the solution of \eqref{eqXMN} is not pathwise unique.
These claims are formulated as Theorem~\ref{thm22},~\ref{thm23} and
\ref{thm24} in Section~\ref{secex}.

We close the introduction with a remark on Theorem~\ref{thm1}. After
rearranging and conditioning on $B^{(2)}$, Theorem~\ref{thm1} says
that for
almost all sample path $w$'s of a Brownian motion, the solution of the next
equation is pathwise unique, hence strong:
%
%
%
\begin{equation}\label{eqXB}
dX_t=\sign(X_t+w_t)\,dB_t,\qquad X_0=0.
\end{equation}
Denote by $H\subset C[0,\infty)$, the set of those deterministic
functions $w$
for which the solution of \eqref{eqXB} is pathwise unique. Then $H$
is not
empty, and the above reasoning gives that it has full measure with
respect to the
Wiener measure on the path space. On the other hand, to construct one such
example not using randomness seems to be difficult. One possible reason
for it is
that $H$ might be small in the sense of category. So the natural question
arises, for which we do not know the answer: is the set $H$ meager,
that is, of
the first Baire category?

\section{Proofs}\label{sec2}
We prove Theorem~\ref{thm3} below, but first we show how to deduce Theorem
\ref{thm2} from Theorem~\ref{thm3}.
\begin{pf*}{Proof of Theorem~\ref{thm2} using Theorem~\ref{thm3}}
We have to show that if $X$ and $X'$ are two solutions
of \eqref{eqXMN}, such that $X_0=X'_0$, then $X=X'$.
We can assume that $X_0=X'_0=0$,
since up to the stopping time $\tau=\inf\{t>0\dvtx X_t=0\}$, the
solution is given by $X_0+\sign(X_0)M_t+N_t$.

So, we can assume that $X_0=X'_0=0$. As indicated in the remark before Theorem
\ref{thm3}, put $U_t=(X_t-X'_t)/2$. Then 
%
%
%
\begin{eqnarray}\label{eqZ}
U_t&=&\frac12\int_0^t\sign(X_s)-\sign(X'_s)\,dM_s\nonumber
\\[-8pt]
\\[-8pt]
&=&\int_0^t \I{X_sX'_s<0}\sign(X_s)\,dM_s=\int_0^t \I{X_sX'_s<0}\,dU_s.
\nonumber
\end{eqnarray}
We obtain \eqref{eqZ}, by observing that
\[
\sign(x)-\sign(x')=
\cases{\displaystyle
2\sign(x),&\quad if $xx'<0$,\cr\displaystyle
\sign(x)+1,&\quad if $x'=0$,\cr\displaystyle
-1-\sign(x'),&\quad if $x=0$,\cr\displaystyle
0,&\quad if $xx'>0$
}
\]
and 
%
%
%
\begin{equation}\label{eqX=0=X}
\int_0^t \I{X_s=0} \bigl(1+\sign(X'_s)\bigr)\,dM_s =\int_0^t \I
{X'_s=0}\bigl(1+\sign(X_s)\bigr)\,dM_s=0.
\end{equation}
To show
\eqref{eqX=0=X} put $\xi_t=\int_0^t \I{X'_s=0}(1+\sign(X_s))\,dM_s$,
and use
$\mathbf{E}(\xi^2_t)\leq\mathbf{E}(\q{\xi}_t)$ combined with
\[
\label{eqq=0}
\q{\xi}_t=\int_0^t \I{X'_s=0}\bigl(1+\sign(X_s)\bigr)^2\,d\q
{M}_s\leq4\int
_0^t \I{X'_s=0}\,d\q{X'}_s=0.
\]
The latter is an easy consequence of the occupation time formula. The other
part of \eqref{eqX=0=X} follows similarly, 
by changing the role of $X$ and $X'$.

We can observe that $X_tX_t'<0$ if and only if
$|X_t-X_t'|>|X_t+X_t'|$, that is,
$|U_t|>|V_t|$, where $V=(X+X')/2$. Hence equation \eqref
{eqZ} is
just another form of 
\eqref{eqUV}.
By definition,
\begin{eqnarray*}
\q{U}_t=\int_0^t \I{X_sX'_s<0}\,d\q{M}_s, \qquad
\q{V}_t=\int_0^t \I{X_sX_s'\geq0}\,d\q{M}_s+\q{N}_t.
\end{eqnarray*}
So $\q{U,V}=0$, that is, $U$ and $V$ are strongly orthogonal, and $V$ dominates~$U$.
By Theorem~\ref{thm3}, $2U=X-X'$ is identically zero, hence $X=X'$.
\end{pf*}
%
%
\begin{remark}\label{remequi}
Observe that any nontrivial example to \eqref{eqUV} can produce an
example showing that the solution of the corresponding perturbed Tanaka
equation is not pathwise unique. Indeed take strongly orthogonal $U,V$
such that \eqref{eqUV} holds and $U$ is not identically zero.
Define
\begin{eqnarray*}
X&=&V+U, \qquad
X'=V-U, \\
Y_t&=&\int_0^t \I{|V_s|\geq|U_s|}\,dV_s, \qquad
W_t=\int_0^t \I{|V_s|< |U_s|}\,dV_s.
\end{eqnarray*}
By enlarging the probability space, one can assume that $Y_t=\xi_t+\xi'_t$,
where $\xi$ and $\xi'$ are strongly orthogonal continuous local
martingales and
$\q{\xi}=\q{\xi'}$. To see this take the DDS Brownian motion $B$ of
$Y$ and
a Brownian motion $B'$ independent from the original $\F_\infty$, and write
$\xi_t=\frac12 
(B+ B')_{\q{Y}_t}$, $\xi'_t=\frac12
(B-B')_{\q{Y}_t}$.
With this choice $U,W,\xi,\xi'$ are pairwise strongly orthogonal.

Finally let
\[
N=W+\xi'\quad\mbox{and}\quad M_t=\int_0^t \sign(X_s)(dU_s+d\xi_s).
\]

The point here is that by \eqref{eqUV},
\[
dU_t=\I{X_tX'_t<0}\sign(X_t)\,dM_t\quad\mbox{and}\quad
d\xi_t=\I{X_tX'_t\geq0}\sign(X_t)\,dM_t,
\]
since $XX'<0$ exactly when $|U|>|V|$.
Hence
\[
dX_t=dW_t+d\xi'_t+d\xi_t+dU_t=dN_t+\sign(X_t)\,dM_t.
\]
Note also that the calculation leading to \eqref{eqX=0=X}, and finally
\eqref{eqZ}, applies with the current definition of $M$, $X$ and
$X'$, since
both $X$ and $X'$ dominate $M$. Hence
\[
\bigl(\sign(X_t)-\sign(X'_t)\bigr)\,dM_t=\I{X_tX'_t<0}\sign(X_t)\,
dM_t=2\,dU_t,
\]
and
\[
dX'_t=dX_t-2\,dU_t=dN_t-\sign(X'_t)\,dM_t.
\]
%
That is, both $X$ and $X'$ solves
\eqref{eqXMN}. Moreover, $N$ dominates $M$ exactly when $V$ dominates $U$,
since
%
%
%
\begin{equation}
\label{eqqNM}
\q{N}=\q{V},\qquad
\q{M}=\q{U}+\q{Y},
\end{equation}
and $V$ dominates $Y$ by definition.
\end{remark}

\subsection{\texorpdfstring{Outline of the proof of Theorem \protect\ref{thm3}}
{Outline of the proof of Theorem 5}}\label{sec21}
In the previous remark we already defined $Y,W$ as
%
%
%
\begin{equation}\label{eqYWdef}
Y_t=\int_0^t \I{|V_s|\geq|U_s|}\,dV_s, \qquad
W_t=\int_0^t \I{|V_s|< |U_s|}\,dV_s.
\end{equation}
Assume that \eqref{eqUV} holds. Then the key feature of $Y$ and
$(U,W)$ is
that they cannot change ``simultaneously.''
One of the simplest examples for two continuous martingales without simultaneous
moving is used in one of the proofs of the arcsine law; see, for
example, Theorem 2.7 of
Chapter VI on page 242 of \cite
{Yor}.
In this proof one splits the Brownian motion $B$ with the formula
\[
B^{+}_t=\int_0^t \I{ B_s>0}\,dB_s,\qquad B^-_t=\int_0^t -\I
{B_s<0}\,dB_s,
\]
and exploits the fact that the two processes $B^+$ and $B^-$ are
linked to each other through the local time of $B$ at level zero, that is,
\[
\inf_{s\leq t}B^{+}_s=\inf_{s\leq t} B^{-}_s =-\frac{1}{2}L^0_t(B)
\qquad
\mbox{for
all $t\geq0$.}
\]
It means that the excursions of $B^+$ and $B^-$ from their running minimum
are interlaced. Heuristically, after each excursion of $B^{+}$ the
value of
the running minimum process decreases with an infinitesimal value. Before
these infinitesimal decrements sum up to a visible change, $B^{-}$
performs some
excursions as well, so the running minimum processes remain synchronized.

Now suppose, contrary
to Theorem~\ref{thm3}, that we have a
nontrivial pair $(U,V)$ of strongly orthogonal, continuous local martingales
satisfying \eqref{eqUV}. Then, similarly as in the
above example, $Y$ and $(U,W)$ are ``linked'' to each other, although the
situation is somewhat more complex.
To describe this link
take the random sets
\[
A^{+}=\{t\dvtx|V_t|>|U_t|\},\qquad
A^{-}=\{t\dvtx|V_t|<|U_t|\}.
\]

Say, $(\sigma,\tau)$ is a connected component of $A^+$. Then $(U,W)$ is
constant on $(\sigma,\tau)$ while the process $Y$ takes a move.\vadjust{\goodbreak}
Then $Y$ stays on one side of $Y_\sigma$, and at the end of the interval,
that is, at $\tau$, it returns to the starting value of the excursion,
that is,
$Y_\tau=Y_\sigma$.

The other case is when $(\sigma,\tau)$ is a component of $A^-$. Then
$Y$ is
constant, and $(U,W)$ makes a move. Since for $t\in(\sigma,\tau)$
we have $|Y_t+W_t|<|U_t|$, the two-dimensional process
$(U,W)$ moves
in the interior of a ``double cone'' until it reaches
the boundary. To be precise this double cone is $C(-Y_\sigma)$, where
\[
C(y)=\{(u,w)\in\real^2\dvtx y-|u|\leq w \leq y+|u|\}.
\]

The best way to think of the above is that the two-dimensional process $(U,W)$
moves in the plane under the constraint that it can not leave the (moving)
double cone
$C(-L_t)$,
where $L_t=Y_{\sigma(t)}$ the value of $Y$ at the last time epoch when
$|Y+W|=|U|$. When $(U,W)$ hits the boundary of $C(-L_t)$, it
has to wait
until the change in $L_t$ enables it to move.

Recall that this is similar to
the way $B$ is obtained from $B^+$ and $B^-$. In the case of $B$, the constraint
is that $B^{+}$ must be in the moving half line $\{x\in\real\dvtx
x\geq\inf_{s\leq t} B^{-}_s\}.$ Since there is a one-sided
condition, both
processes have only excursions from the running minimum.

By similar reasoning,
when $(U,W)$ hits the polyline $\{(u,w)\in\real^2\dvtx w=-L_t+|u|\}$, then
$-L$ is locally increasing, as $(U,W)$ pushes the double cone $C(-L_t)$ upward
on the plane. Actually, $L$ locally follows the running minimum
of $Y$, and as in the case of $B^\pm$, the changes in $L$ can be
described as
the changes of a local time process; see Lemma~\ref{lemTV} below.

The other case,
that is, when $(U,W)$ hits the polyline $\{(u,w)\in\real^2\dvtx
w=-L_t-|u|\}$
differs only in the direction of changes. In this regime, $(U,W)$ tries
to push
downward the cone on the plane, and therefore $-L$ is decreasing. Then $Y$
performs excursions below the
actual value of $L$, and $L$ locally follows the running maximum of $Y$.

The above reasoning is made precise in Lemma~\ref{lemTV} and yields
that $L$
is a linear combination of local time processes, whence it has a continuous
sample path with a locally bounded variation.

The end of our argument is that immediately after the
moment that $U$ leaves the origin, the total variation of $L$
becomes infinite. Since $L$ has locally bounded variation,
this clearly implies that $U$
is identically zero and, in other words, is trivial.

To do this last step, we only use that under the assumptions of Theorem~\ref{thm3}
the local martingales $U,W$ are strongly orthogonal, $W$
dominates $U$ and $(U_t,W_t)$ remains in the double cone $C(-L_t)$ for all
$t$, that is, $W-|U|\leq-L \leq W+|U|$.
To fix ideas let us discuss here the simplest
case; that is, assume that $(U,W)$ is a two-dimensional Brownian
motion, and $L$ is
continuous process such that $W-|U|\leq-L \leq W+|U|$.
Denote by $\TV_t$ 
the total variation of $L$ on $[0,t]$.
Next we give the reason why $\TV_t$ becomes infinite immediately after
starting.

During each excursion of $|U|$ away from zero, the process $\TV$
increases. Take one such excursion which is performed on the time interval
$I=[s,t]$. Then $-L_s=W_s$ and $-L_t=W_t$ since $U_s=U_t=0$. The
increment of $\TV$ on $I$ can be estimated as
$\TV_t-\TV_s\geq|L_t-L_s|=|W_t-W_s|$. Here
$(W_t-W_s)/\sqrt{t-s}$ is
a standard normal variable, by the
independence of $U$ and $W$.
Moreover, if we take the usual measurable
enumeration of the excursions, then the corresponding normal variables are
independent of each other and also of $U$. Hence we have a lower bound for
$\TV_t$ in the form
%
%
%
\begin{equation}\label{eqsum}
\sum_n \sqrt{|I_n|}|\eta_n|,
\end{equation}
where $\{I_n\dvtx n\geq0\}$ is the enumeration of excursion intervals ending
before~$t$, and the variables $|\eta_n|$ are i.i.d., with
positive expectation,
independent of the sequence $|I_n|$.
By a characterization of Brownian local time we have $\sum_n
\sqrt{|I_n|}=\infty$ a.s.,
and this implies immediately that \eqref{eqsum} is also almost surely
infinite. This shows that $\TV_t=\infty$ for $t>0$.

With some modification the above reasoning also applies to $U,W$ and
$L$ in
the general case.

\subsection{\texorpdfstring{Details of the proof of Theorem \protect\ref{thm3}}
{Details of the proof of Theorem 5}}\label{sec22}
Throughout this section, for $t\geq0$ put
\[
\sigma(t)=\sup\{s\in[0,t]\dvtx|U_s|=|V_s|\}.
\]
$\sigma(t)$ is the last point before $t$
where $|V|=|U|$ holds. The process $\sigma$
is increasing, right continuous and adapted.
It starts at zero, since by assumption $U_0=V_0=0$.

Next, $Y,W$ are defined by formula \eqref{eqYWdef} and $L$ by
%
%
%
\begin{equation}\label{eqLdef}
L_t=Y_{\sigma(t)}.
\end{equation}
The reasoning outlined in the preceding section is accomplished by
proving two
lemmas below. Lemma~\ref{lemTV} gives that
$L$ has continuous sample path with locally bounded variation. Lemma
\ref{lemL} applies to $L$ by Proposition~\ref{proporder} and
formalizes the
argument at the end of the heuristic argument. It shows that the assumption
that $U$ is not identically zero would lead to a contradiction proving Theorem
\ref{thm3} completely. The proof of Lemma~\ref{lemL} uses two more
proposition and a slight addition to Knight's theorem; see Lemma
\ref{lemknight}.

%
%
\begin{lemma}\label{lemTV} Let $U,V$ be continuous semimartingales
satisfying \eqref{eqUV} and
$L$ as above. Then $L$ is a linear combination of local time processes, and
hence it is of bounded variation on compact intervals. To be precise,
\[
2L_t=L^0_t(|U|+V)-L^0_t(|U|-V),
\]
where $L^x(X)$ denotes the local time process of $X$ at level $x$.
\end{lemma}
\begin{pf} Put $\xi=\med(V+U,V-U,0)$, where med denotes the median
of its three argument. Then $\xi_t$ follows the\vadjust{\goodbreak} trajectory of $V+U$ if
it is
in the middle, that is, when $UV<0$ and $|V|>|U|$. It
follows the
changes of $V-U$ when $UV>0$ and $|V|>|U|$ and stays at zero when
$|V|<|U|$. When $\xi$ switches between the above regimes, the
corresponding local time process increases. So apart form the local time
changes,
$\xi_t$ follows the changes in $Y$, since the other two processes
$W,U$ are
locally constant on $\{t\dvtx|V_t|>|U_t|\}$.

We obtained that $\xi_t=Y_t-L'_t$, where
$L'_t$ is from the local time components. Now if $\sigma(t)=t$, that
is, 
$|V_t|=|U_t|$ then we have $\xi_t=0$. Hence
$L_t=Y_{\sigma(t)}=L'_{\sigma(t)}$. This gives that $L$ is of locally
bounded variation.

To carry out this program observe that
%
%
%
\begin{eqnarray}\label{eqxi}
\xi_t&=&\med(V+U,V-U,0)
=(V_t+|U_t|)\wedge0 +(V_t-|U_t|)\vee0 \nonumber
\\[-8pt]
\\[-8pt] &=&
(|U_t|+V_t)\wedge0-(|U_t|-V_t)\wedge0.
\nonumber
\end{eqnarray}

For the first term, the Tanaka formula gives that
\[
d(|U|_t+V_t)\wedge0= 
\I{V_t\leq-|U_t|}\,d(|U_t|+V_t)-\tfrac{1}{2}\,d L^{0}_t(|U|+V).
\]
Note that since $U,V$ satisfies \eqref{eqUV},
and the support of
$dL^0_t(U)$ is the null set of $U$, the right-hand side
simplifies to
\[
d(|U|_t+V_t)\wedge0= 
\I{V_t \leq-|U_t|}\,dV_t+ \I{V_t \leq0}\,dL^0_t(U)-\tfrac{1}{2}\,d
L^0_t(|U|+V).
\]
Similar calculation for the second term in \eqref{eqxi} yields
\[
d(|U_t|-V_t)\wedge0=
-\I{V_t \geq|U_t|}\,dV_t+ \I{V_t \geq0}\,dL^0_t(U)-\tfrac{1}{2}\,d
L^0_t(|U|-V).
\]
Hence
\[
d\xi_t=\I{|V_t|\geq|U_t|}\,dV_t-
\sign_0(V_t)\,dL^0_t(U)+
\tfrac{1}{2}\,d L^{0}_t(|U|-V)-\tfrac{1}{2}\,d L^0_t(|U|+V),
\]
where $\sign_0=\I{x>0}-\I{x<0}$.

The first term on the right is simply $dY_t$ by definition.
The support of $dL^0(U)$ is a subset of $\{t\geq0\dvtx V_t=U_t=0\}$,
since on the components of its complement either $U$ is nonzero or $U$
is locally constant. Hence the second term on the right is zero.

After these simplifications, using that $\xi_0=0$, we obtain
\begin{eqnarray*}
\xi_t&=&Y_t-L'_t ,\\
L'_t&=&\tfrac{1}{2}L^0_t(|U|+V) -\tfrac{1}{2}L^{0}_t(|U|-V).
\end{eqnarray*}
To finish the proof use that $\xi_{\sigma(t)}=0$ for all $t\geq0$;
that is,
$L_t=L'_{\sigma(t)}$ and
that $(\sigma(t),t)$ is disjoint from the support of all the involved local
time processes, and hence $L'_t=L'_{\sigma(t)}$.
\end{pf}

Note that the formula, obtained for $\xi$, is the special case of
the general formula for ranked semimartingales 
proved recently in~\cite{MR2428716}.

%
%
\begin{proposition}\label{proporder}
Let the continuous semimartingales $U,V$ satisfy \eqref{eqUV} and $L,W$
defined by \eqref{eqYWdef} as
above. Then $|L_t+W_t|\leq|U_t|$ for all $t\geq0$.\vadjust{\goodbreak}
\end{proposition}
\begin{pf}By definition at $s=\sigma(t)$ we have
$|L_s+W_s|=|U_s|$. It is enough to consider the case when $s<t$,
since otherwise we are done. On the interval $(s,t]$ either
$|V|>|U|$ or $|V|<|U|$.
In the first case,
$W$, $U$ and $L$ are constant on $[s,t]$,
and we get the statement with equality. In the
second case, $Y$ is constant on $[s,t]$;
hence $L_t=Y_t$, and the statement follows, since then
$|L_t+W_t|=|Y_t+W_t|=|V_t|<|U_t|$.
\end{pf}

%
%
\begin{lemma}\label{lemL}Let $U$ and $W$ be strongly orthogonal continuous
local martingales starting from zero. Assume that $W$ dominates $U$, and
for the continuous process $L$, we have $|L_t+W_t|\leq|U_t|$ for
$t\geq0$.

Then, the total variation process $(\TV_t)\tgo$ of $L$ satisfies
%
%
%
\begin{equation}\label{eqlemL}
\TV_t
=
\cases{\displaystyle
0,&\quad if $U_s=0$ for $s\leq t$,\cr\displaystyle
\infty,&\quad otherwise.
}
\end{equation}
That is, immediately after $U$ leaves the origin, $\TV$
becomes infinite.
\end{lemma}
\begin{remark*}
By enlarging the probability space if necessary, we may assume that
both $U$
and $W$ are divergent martingales. Indeed, enlarge a probability space
with a
two-dimensional Brownian motion $B=(B^{(1)},B^{(2)})$, independent of
$\F_\infty$. Fix a $T>0$, and define $\bU,\bW$ and a new filtration
$(\bF_t)\tgo$ with the formulas
\begin{eqnarray*}
\bF_t&=&\F_t\vee\F^B_t,\\
\bU_t&=&U_{t\wedge T}+B^{(1)}_t-B^{(1)}_{t\wedge T},\\
\bW_t&=&W_{t\wedge T}+B^{(2)}_t-B^{(2)}_{t\wedge T}.
\end{eqnarray*}
Now we can define $\bL$ to satisfy the assumption of Lemma \ref
{lemL} in many
ways. One possibility is to define $\tau$ be the first time after $T$ when
$|L+\bW|$ meets $|\bU|$. Up to $\tau$ the process $\bL$
is the same as
the stopped process $L_{t\wedge T}$. After $\tau$, the process $\bL$ follows
the changes of
either $-\bW-|\bU|$ or
$-\bW+|\bU|$ according to which hits before the level $L_T$.
Formally one
could define $\bL$ as
\begin{eqnarray*}
\xi^\pm_t&=&-\bW_t\pm|\bU_t|,\\
\tau^\pm&=&\inf\{t\geq T\dvtx L_T=\xi^\pm_t\},\\
\bL_t&=&L_{t\wedge T}+ \I{\tau^+<\tau^-}(\xi^+_t-\xi^+_{t\wedge
\tau^+})+
\I{\tau^+\geq\tau^-}(\xi^-_t-\xi^-_{t\wedge\tau^-}).
\end{eqnarray*}
Using the independence of $B$ and $\F_\infty$, it follows that $\bU$ and
$\bW$ are orthogonal continuous local martingales in $(\bF_t)\tgo$. By
construction $\bU$, $\bW$ are divergent, $\bW$ dominates $\bU$, the process
$\bL$ has continuous sample paths and
$|\bL_t+\bW_t |\leq|\bU_t |$ almost surely for all $t$.

Now, if the statement of Lemma~\ref{lemL} holds for the triple
$(\bL,\bU,\bW)$,
then it
also holds for $(L,U,W)$,
provided that $t$ in \eqref{eqlemL} is smaller
than $T$. Since $T>0$ was arbitrary, the Lemma follows from the special case
when $U$ and $W$ are divergent.
\end{remark*}
\begin{notation*} To shorten formulas, we use 
$\Delta_I X$ for the change of the process $X$ on the interval $I$.
\end{notation*}
\begin{pf*}{Proof of Lemma~\ref{lemL}}
According to the previous remark, we may and do assume that both
$U$ and $W$ are divergent.
For $\eps>0$, let
$\tau(\eps)=\tau(U,\eps)=\inf\{t>0\dvtx L^0_t(U)>\eps\}$.
Since $U$ is a divergent local
martingale, $\tau(\eps)$ is finite almost surely. Clearly it is
enough to show
that $\TV_{\tau(\eps)}=\infty$ for any fixed $\eps>0$.
In the first part of the proof, we fix a ``typical''
$\omega\in\Omega$, but in the notation it is suppressed.

Let $\z=\{t\dvtx U_t=0\}$ denote the null set of $U$.
Also, let $\cC$ denote the collection of connected components of
$\{t\dvtx U_t\neq0\}$ and $\cC(\eps)=\{I\in\cC\dvtx I\subset
[0,\tau(\eps)]\}$.
Since $U$ is divergent,
for a typical $\omega$, that is,
with probability one, $\cC$ and $\cC(\eps)$ has infinitely many elements.

Next, since $W$ dominates $U$, there is a $c>0$ such that $d\q
{U}_t\leq
cd\q{W}_t$,
that is,
the increase of $\q{W}$ on any interval $I$ is at least
$\Delta_I\q{U}/c$. Hence,
\[
\gamma_{(a,b)}=\inf\biggl\{t\geq a\dvtx\q{W}_t-\q{W}_a=\frac{\q
{U}_b-\q{U}_a}{2c}\biggr\}
\]
defines a time-point in $(a,b)$.

If $I=(a,b)\in\cC$, then $U_a=0$ and $L_a
=-W_a$ by assumption.
Also by our assumption,
$|L_s+W_s|\leq|U_s|$ for $s=\gamma_{(a,b)}$, hence by the triangle
inequality,
\[
\Delta_I \TV\geq|L_s-L_a|\geq|W_s-W_a|-|U_s|\geq
|W_s-W_a|-\sup_{u\in(a,b)}|U_u|.
\]

This gives
\[
\TV_{\tau(\eps)}\geq\sum_{I\in\cC(\eps)} 
(\Delta_I\q{U})^{1/2}
(|\xi_I|-\eta_I)^+,
\]
where $(x)^+=0\vee x$ is the positive part of $x$ and for $I=[a,b]$
\begin{eqnarray*}
\label{eqx}
\xi_{I}=\frac{1}{(\Delta_I\q{U})^{1/2}}(W_{\gamma_I}-W_a)
, \qquad
\eta_{I}=\frac{1}{(\Delta_I\q{U})^{1/2}}\sup_{s\in I}|U_s|.
\end{eqnarray*}

We claim the following:
%
%
\begin{proposition}\label{lemindep}
There is a measurable enumeration of the random collection of intervals
$\cC(\eps)=
\{I_n\dvtx n\geq1\}$
such that
$(\xi_{I_n},\eta_{I_n})$, $n\geq1$ is an i.i.d. sequence independent of
$\A=\sigma(\{\Delta_{I_n}\q{U}\dvtx n\geq1\})$.
Moreover, $\mathbf{E}((|\xi_{I_n}|-\eta_{I_n})^+)$ is positive
and finite.
\end{proposition}

%
%
\begin{proposition}\label{lemsum}
\[
\sum_{I\in\cC(\eps)} 
(\Delta_I\q{U})^{1/2}=\infty\qquad\mbox{almost surely.}
\]
\end{proposition}

The end of the proof is then rather straightforward. For independent
nonnegative random variables\vadjust{\goodbreak} $X_1,X_2,\ldots,$ the sum $\sum_n X_n$ is
finite if and only if $\sum_n \mathbf{E}(X_n\wedge1)<\infty$; see
Proposition 3.14 of
\cite{Kallenberg}. When $\mathbf{E}(X_n)<\infty$ for all $n$ and $X_n/\mathbf{E} (X_n)$ is an i.i.d.
sequence the
truncation can
obviously be dropped.
Thus, conditioning first on $\A$, we can apply
this result to
$X_n=(\Delta_{I_n}\q{U})^{1/2}(|{\xi_{I_n} }|-\eta
_{I_n})^+$, by
Proposition~\ref{lemindep}. Since by Proposition~\ref{lemsum},
$\sum_n
\mathbf{E}(X_n|\A)=\infty$ almost surely, the lower bound for
$\TV_{\tau(\eps)}$ is infinite almost surely.
\end{pf*}

Proposition~\ref{lemindep} is probably the most delicate part of the
proof. It is based on a slight extension to Knight's theorem, Lemma
\ref{lemknight}. For a divergent continuous local martingale $M$
starting at
zero, we say that $\beta$ is the DDS Brownian motion of $M$ if
$\beta_t=M_{\rho(t)}$ where
\[
\rho(t)=\inf\{s>0\dvtx\q{M}_s>t\}.
\]
Then $\beta$ is a Brownian motion; see Chapter V in~\cite{Yor}.

To prove Proposition~\ref{lemindep} we use the next statement whose
proof is deferred to the end of the section.
%
%
\begin{lemma}\label{lemknight}
Let $M,N$ be divergent, continuous local martingales in the
filtration
$(\F)\tgo$. Assume that $M$ and $N$ are strongly orthogonal. Denote
$\beta$
the DDS Brownian--motion of $N$. Then $M$ is a local martingale in the
filtration $(\bF_t)\tgo$, where $\bF_t=\bigcap_{s>t}(\F_s\vee
\sigma
(\beta))$.
\end{lemma}

\begin{pf*}{Proof of Proposition~\ref{lemindep}}
Denote by $\beta$ the DDS Brownian motion of $U$, let
$\bF_t=\bigcap_{s>t}(\F_s\vee\sigma(\beta))$. Then by Lemma
\ref{lemknight} the process $W$ is a local martingale in the larger
filtration $\bF$ as well.

Let $\z(\beta)$ be the null set of $\beta$, and denote by $\cC
(\beta)$ the connected
components of the complement of $\z(\beta)$ and
$\cC(\beta,\eps)=\{I\in\cC\dvtx I\subset[0,\tau(\beta,\eps)]\}
$ where
$\tau(\beta,\eps)=\inf\{t>0\dvtx L^0_t(\beta)>\eps\}$. Besides, let
$\sigma(\z(\beta))=\sigma\{C_{s,t}\dvtx0\leq s\leq t\}$ the smallest
$\sigma$-algebra containing the events $C_{s,t}=\smallset{[s,t]\cap
\z(\beta)=\varnothing}$.

Then we define the enumeration of
$\cC(\eps)$ based on the usual $\sigma(\beta)$ measurable enumeration
$\{J_n\dvtx n\geq1\}$ of $\cC(\beta,\eps)$.
Indeed, $J_n=(a_n,b_n)$ with some
$\sigma(\beta)$ measurable random time $a_n,b_n$; then let
$I_n=(\rho(a_n),\rho(b_n))$, where $\rho(t)=\inf\{s>0\dvtx\q
{U}_s>t\}$.
The point here is that the random times $\rho(a_n),\break\rho(b_n), \gamma
_I$ are
stopping times in the filtration $(\bF_t)\tgo$.

This implies that for any finite collection $F\subset\card$ the random
variables $\{\xi_{I_n}\dvtx n\in F\}$ are independent also from each
other and
of $\bF_0$. To see this, we can define the simple $\bF$-predictable process
\[
H_t=\sum_{n\in F} \frac{\alpha_n}{({\Delta_{I_n}\q
{U}})^{1/2}} \I{\rho(a_n)<t\leq
\gamma_{I_n}}
\]
with
$\alpha_n\in\real$. Then $H\cdot W$ has uniformly bounded quadratic
variation $\q{H\cdot W}_\infty=\sum_{n\in F} \alpha_n^2/2c$,
which is deterministic.
Using that $\operatorname{exp}\{i H\cdot W +\frac{1}{2}\q{H\cdot W}\}
$ is a bounded
martingale, we
get $\mathbf{E}(\operatorname{exp}\{i (H\cdot W)_\infty+\frac
{1}{2}\q{H\cdot W}_\infty\}|\bF_0)=1$. This
yields the joint conditional characteristic function of
$\{\xi_{I_n}\dvtx n\in F\}$, given $\bF_0$
\[
\mathbf{E}\biggl(\operatorname{exp}\biggl\{i\sum_{n\in F}\alpha_n\xi
_{I_n}\biggr\}\Big|\bF
_0\biggr)=
\operatorname{exp}\biggl\{-\frac12\sum_{n\in F} \frac{\alpha_n^2}{2c}\biggr\}.
\]
That is,
$\{\xi_{I_n}\dvtx n\geq1\}$ is an i.i.d. sequence which is
independent from
$\bF_0$, and the common law is normal with expectation 0 and variance $1/2c$.

$\eta_{I_n}$ is calculated from the normalized excursions of $\beta$
on $J_n$;
hence they form an i.i.d. sequence measurable with respect to
$\sigma(\beta)\subset\bF_0$ and independent of $\sigma(\z(\beta
))$; see
\cite{MR0345224}, Section 2.9.

Finally, $\Delta_{I_n}\q{U}$ is the length of $J_n$, and hence it is
$\sigma(\z(\beta))$ measurable.

Putting these pieces together, we obtain that $(\xi_{I_n},\eta
_{I_n})$ is an
i.i.d. sequence independent of $\sigma(\z(\beta))\supset\A$. The
claim that
$\mathbf{E}((|\xi_{I_n}|-\eta_{I_n})^+)>0$ and finite is
obvious from
the joint law of $(\xi_{I_n},\eta_{I_n})$.
\end{pf*}

\begin{pf*}{Proof of Proposition~\ref{lemsum}}
With the notation introduced in the proof of Proposition
\ref{lemindep}, we can reformulate the statement. Using $\beta$ the DDS
Brownian motion of $U$, we have to show that for $\eps>0$,
\[
\sum_{J\in\cC(\beta,\eps)}
|J|^{1/2}=\infty\qquad\mbox{almost surely},
\]
which follows from a characterization of the local time.

Indeed, let $n_k$ be the number
of intervals in $\cC(\beta,\eps)$ longer than $2^{-k}$.
Then the limit $\lim_{k\to\infty}2^{-k/2}n_k$ almost surely
exists and is positive; it is $\sqrt{2/\pi} L_{\tau(\beta,\eps
)}^0(\beta)$;
see, for example,~\cite{Yor}, Proposition (2.9), Chapter XII. From
this, the
statement follows using elementary analysis, since
\begin{eqnarray*}
 \sum_{J\in\cC(\beta,\eps)}
|J|^{1/2}&\geq&
\sum_{k=1}^\infty2^{-k/2} (n_k-n_{k-1}) \\
 &=&
-\frac{n_0}{\sqrt{2}}+(1-2^{-1/2})\sum_{k=1}^\infty2^{-k/2} n_k
=\infty.
\end{eqnarray*}
\upqed
\end{pf*}
%
%
\begin{remark}\label{remalpha}
With obvious modification the previous calculation also gives that for
$\alpha> 1/2$, we have $\sum_{J\in\cC(\beta,\eps)} |J|^{\alpha
}<\infty$
almost surely.
\end{remark}

\begin{pf*}{Proof of Lemma~\ref{lemknight}}
$M$ is a divergent continuous local martingale, denote $B$ its DDS Brownian
motion. That is $B_t=M_{\rho(t)}$ with the time-change $\rho$ associated
with the quadratic variation of $M$.
Then $B$ is a Brownian motion in the time-changed
filtration $(\G_t=\F_{\rho(t)})\tgo$, and $(\q{M}_t)\tgo$ is continuous
time-change in the filtration $(\G_t)\tgo$.

We actually show that $B$ is a martingale in the filtration
$(\bG_t)\tgo$, where $\bG_t= \G_t\vee\sigma(\beta)$; that is,
for $0\leq t\leq
s$, we have $\mathbf{E}(B_s-B_t|\bG_t)=0$.\vadjust{\goodbreak}

To see this, fix $t\geq0$, and observe first that the time-shifted processes
$(M_{\rho(t)+u}-M_{\rho(t)})_{u\geq0}$ and
$(N_{\rho(t)+u}-N_{\rho(t)})_{u\geq0}$ are divergent,
continuous local martingales in the time-shifted filtration
$(\F_{\rho(t)+u})_{u\geq0}$. Their DDS Brownian motions are given by
$(B_{t+s}-B_t)_{s\geq0}$ and $(\beta_{\eta(t)+s}-\beta_{\eta
(t)})_{s\geq0}$,
respectively, where $\eta(t)=\q{N}_{\rho(t)}$.

By Knight's theorem (see Theorem 1.9 of Chapter 5 in~\cite{Yor}), the processes
$(B_{t+s}-B_t)_{s\geq0}$ and $(\beta_{\eta(t)+s}-\beta_{\eta
(t)})_{s\geq
0}$ constitute
a two-dimensional Brownian motion in its own filtration and, with a little
extension of the original statement, independent of $\G_t=\F_{\rho(t)}$.
The independence follows from considering the conditional law given $\G_t$.

Next, note that
\[
\bG_t=\G_t\vee\sigma(\beta)=
\G_t\vee\sigma\bigl(\bigl\{\beta_{\eta(t)+s}-\beta_{\eta(t)}\dvtx
s\geq0\bigr\}\bigr),
\]
since $\eta(t)=\q{N}_{\rho(t)}$ is $\F_{\rho(t)}=\G_t$ measurable.

Then, the three $\sigma$-algebras:
$\A_1=\sigma(\{\beta_{\eta(t)+s}-\beta_{\eta(t)}\dvtx s\geq0\})$,
$\A_2=\sigma(\{B_{t+s}-B_{t}\dvtx s\geq0\})$ and $\G_t$ are
independent.
For $s\geq0$ this gives that $B_{t+s}-B_t$ is independent from
$\A_1\vee\G_t=\bG_t$, and $\mathbf{E}(B_{t+s}-B_t|\bG_t)=\mathbf
{E}(B_{t+s}-B_t)=0$, showing
that $B$ is not only a $\G$ Brownian motion, but also a $\bG$
Brownian motion.

Since $M$ is obtained from $B$ with a continuous $\bG$-time-change
$(\q{M}_t)\tgo$, it is a local martingale in the filtration
$
(\bG_{\q{M}_t})\tgo$ and also in its right continuous hull. Now
\[
\bG_{\q{M}_t}\supset\G_{\q{M}_t}\vee\sigma(\beta)\supset\F
_t\vee\sigma(\beta)
\]
finishes the proof.
\end{pf*}

\section{Examples, showing that domination is necessary}\label
{sec3}\label{secex}
The aim of this section is to show that we cannot drop the domination
condition in Theorems~\ref{thm2} and~\ref{thm3} completely.
It is enough to give an example showing that without domination, Theorem
\ref{thm3} does not hold, since by Remark~\ref{remequi}, it also
provides an
example for Theorem~\ref{thm2}.

First we describe $L$ in terms of $U,W$ in a way which is invariant under
time-change. This characterization is similar in spirit to the reflection
lemma of Skorohod.

\renewcommand\thelonglist{\roman{longlist}}
\renewcommand\labellonglist{(\thelonglist)}
%
%
\begin{lemma}\label{lemskorohod} Let $f,g,h\dvtx[0,\infty)\to\real
$ be
continuous
functions satisfying the following properties:
\begin{longlist}[(iii)]
\item\label{propLit1} $f(0)=h(0)=g(0)$;
\item\label{propLit2} $f\leq h\leq g$;
\item\label{propLit3} $h$ is locally nondecreasing on
$\smallset{g\neq h}$ and locally nonincreasing on
$\smallset{h\neq f}$. That is, for $s\leq t$ if $g\neq h$ on $(s,t)$, then
$h(s)\leq h(t)$, and if $h\neq f$ on $(s,t)$, then $h(s)\geq h(t)$.
\end{longlist}
Then
\[
h(t)=
F(d(t),t)=G(d(t),t),\vadjust{\goodbreak}
\]
where
\begin{eqnarray*}
F(s,t)&=&\max\{f(x)\dvtx x\in[s,t]\},\\
G(s,t)&=&\min\{g(x)\dvtx x\in[s,t]\},\\
d(t)&=&\sup\{s\leq t\dvtx F(s,t)\geq G(s,t)\}.
\end{eqnarray*}
\end{lemma}

In plain words, to calculate $h(t)$ go backward starting at $t$ on the
graph of $f$ and $g$ until there is common value in the
range swept by these functions. The first such value is $h(t)$.
%
%
%
We remark that with obvious modifications, Lemma~\ref{lemskorohod}
extends the explicit formula obtained in~\cite{MR2349573} for the two-sided
reflection map on $D[0,\infty)$.
%

\begin{pf*}{Proof of Lemma~\ref{lemskorohod}}
Define $t^g$ and $t^f$ the last time before $t$, when $g=h$ or $f=h$,
respectively; that is, $t^g=\max\{s\in[0,t]\dvtx g(s)=h(s)\}$ and
$t^f=\max\{s\in[0,t]\dvtx f(s)=h(s)\}$.
We can assume that $t^f\leq t^g$; the
other case is obtained by considering $-g\leq-h \leq-f$.

By our assumption (\ref{propLit3}) the function $h$ is
nonincreasing on $(t^f,t)$, and nondecreasing on $(t^g,t)$. 
Since $t^f\leq t^g\leq t$, we have that
$h(s)=h(t)=h(t^g)=g(t^g)$ for all $s\in[t^g,t]$ and also that
$h(t^f)=f(t^f)\geq h(s)\geq h(t)$ for $s\in[t^f,t]$.
Thus
\[
h(t)=\min_{s\in[t^f,t]}h(s) \leq G(t^f,t)\leq g(t^g)=h(t),
\]
that is, $h(t)=G(t^f,t)\leq F(t^f,t)$. By definition, $d(t)\geq t^f$.
On the other hand, $d(t)\leq t^g$ follows from the
fact that if $s\in(t^g,t)$, then $f(s)<h(t)<g(s)$.

Since $F(d(t),t)=G(d(t),t)$ by definition,
we obtain that $h(t)=G(t^f,t)\leq G(d(t),t)\leq G(t^g,t)=h(t)$
and $h(t)=F(d(t),t)=G(d(t),t)$.
\end{pf*}
%
%
\begin{corollary}\label{corh}
Let $f,g\dvtx[0,\infty)\to\real$ be continuous functions and assume that
$f(0)=g(0)$ and $f\leq g$. Then there is a unique continuous
function denoted by
$\bL(f,g)$ such that (\ref{propLit1}), (\ref{propLit2}) and
(\ref{propLit3}) holds for $f\leq h=\bL(f,g)\leq g$.
\end{corollary}
\begin{remark*}
The function $(t,f,g)\mapsto\bL_t(f,g)$ is clearly predictable; see
\cite{Yor}, Chapter IX, for definition.
\end{remark*}

%
%
\begin{corollary}\label{corbL}
Assume that $U,V$ satisfies \eqref{eqUV} and $L,W$ are defined as above.
Then $L_t=\bL_t(-W-|U|,-W+|U|)$.
\end{corollary}
\begin{pf}
By Proposition~\ref{proporder} $W-|U|\leq-L \leq W+|U|$
and by Lemma
\ref{lemTV}, $L$ is continuous, nonincreasing on $|U|+L+W\neq
0$ and
nondecreasing on $|U|-L-W\neq0$.
\end{pf}

Finally, we have the following result which will be proved below in Section~\ref{secproofl15}.
%
%
\begin{lemma}\label{lembZbW}
There is a two-dimensional local martingale $(\bU,\bW)$ on some filtered
probability space such that:
\begin{longlist}
\item\label{lembZbWit1} $\bU$ and $\bW$ are strongly orthogonal;
\item\label{lembZbWit2} $d\q{\bU}\approx d\q{\bW}$ almost
surely, that is,
the random measures induced by the changes of $\q{\bU}$ and $\q{\bW
}$ are
equivalent;
\item\label{lembZbWit3} $\bL=\bL(-\bW-|\bU|,-\bW+|\bU|)$
has locally bounded variation;
\item\label{lembZbWit4} $\bU$ and $\bL$ are divergent.
\end{longlist}
\end{lemma}

Let $(\bU,\bW)$ from Lemma~\ref{lembZbW} and
$\bL_t=\bL_t(-\bW-|\bU|,-\bW+|\bU|)$. 
Then $|\bL_t+\bW_t |\leq|\bU_t |$ for $t\geq0$. We can
assume that $\bW$
is a Brownian motion by applying an appropriate time-change; the proof
is actually formulated in this way. So assume for the moment that
$\q{\bW}_t=t$.
Observe that by Lemma~\ref{lemL} the Brownian motion $\bW$ cannot dominate
$\bU$ on any intervals of the form $[0,t]$, with $t>0$.
Since $\q{\bU}$ is equivalent with $\q{\bW}$, that is, with the Lebesgue
measure, we can write it a-
$\q{\bU}_t=\int_0^t Q_s\,ds $. The nondomination property means that
$\esssup_{s\in[0,t]} Q_s=\infty$ almost surely for all $t>0$.

%
%
\begin{proposition}\label{prop18}
$(\bU,\bL+\bW)$ fulfills \eqref{eqUV}, that is,
\[
d\bU_t=\I{|\bL_t+\bW_t|<|\bU_t|}\,d\bU_t,\qquad\bU
_0=\bW
_0=\bL_0=0.
\]
Moreover decomposition 
\eqref{eqYWdef} gives back $\bL$
and $\bW$, that is,
\[
\bL_t=\int_0^t \I{|\bL_s+\bW_s|\geq|\bU_s|} \,d(\bL
_s+\bW_s),\qquad
\bW_t=\int_0^t \I{|\bL_s+\bW_s|<|\bU_s|}\, d(\bL_s+\bW_s),
\]
and $\bL_{\bsigma(t)}=\bL_t$, where $\bsigma(t)=\sup\{s\leq
t\dvtx|\bU_s|=|\bL_s+\bW_s|\}$.
\end{proposition}
\begin{pf} Since $|\bL+\bW|\leq|\bU|$, to show that
$\bU,\bL+\bW$ satisfies \eqref{eqUV}, we only need that
%
%
%
\begin{equation}\label{eqxi=0}
\xi_t=\int_0^t \I{|\bL_s+\bW_s|=|\bU_s|}\,d\bU_s=0.
\end{equation}
This follows similarly as \eqref{eqX=0=X} above as
$\mathbf{E}(\xi^2_t)\leq\mathbf{E}(\q{\xi}_t)$ and the latter can
be estimated
using the
orthogonality of $\bU$ and $\bW$ by
%
%
%
\begin{equation}\label{eqbZ}
\qquad\int_0^t\I{|\bL_s+\bW_s|=|\bU_s|}\,d\q{\bU}\leq
\int_0^t\I{|\bL_s+\bW_s|=|\bU_s|}\,d\q{|\bL+\bW|-|\bU|}_s=0
\end{equation}
%
by the occupation time formula. The same applies if we integrate with
respect to $\bW$ in \eqref{eqxi=0}.
Thus
\[
\int_0^t \I{|\bL_s+\bW_s|\geq|\bU_s|}\,d(\bL_s+\bW_s)=
\int_0^t \I{|\bL_s+\bW_s|=|\bU_s|}\,d\bL_s=\bL_t.
\]
In the last step
we used that $\bL$ is locally constant on $\bU\neq\bL+\bW$;
cf. Property~(\ref{propLit3}) of Lemma~\ref{lemskorohod}.
This proves the first part of the decomposition formula.
The second part, that is, the formula for $\bW$, obviously follows.

Finally, $(\bsigma(t),t)\subset\{s\dvtx|\bU_s|\neq|\bL_s+\bW
_s|\}$,
hence $\bL$ is constant on $[\bsigma(t),t]$ and $\bL_{\sigma
(t)}=\bL_t$.
\end{pf}

Application of Lemma~\ref{lemTV} proves the next representation of
$\bL$.
%
%
\begin{corollary}\label{cor2L}
%
%
%
\begin{equation}\label{eq2L}
2\bL_t=L^0_t\bigl(|\bU|+(\bL+\bW)\bigr)-L^0_t\bigl(|\bU|-(\bL
+\bW)\bigr).
\end{equation}
\end{corollary}

%
%
\begin{corollary}\label{cortechnical}
%
%
%
\begin{equation}\label{eqL=0}
\int\I{\bU_t=0}\bigl(dL^0_t\bigl(|\bU|+(\bL+\bW)\bigr
)-dL^0_t\bigl(|\bU|-(\bL+\bW)\bigr)\bigr)=0.
\end{equation}
\end{corollary}
\begin{pf}
We use that for a nonnegative continuous semimartingale $X$, we have
\[
\frac12 L^0_t(X)=\int_0^t \I{X_s=0}\,dX_s.
\]
We apply it for $X=|\bU|+(\bL+\bW)$. Using that $\I{\bU=0}\I
{X=0}=\I{\bU=0}$, we obtain that
\begin{eqnarray*}
&&\int_0^t\I{\bU_s=0}\,dL^0_s\bigl(|\bU|+(\bL+\bW)\bigr)\\
&& \qquad =
2\int_0^t\I{\bU_s=0}\,d\bigl(|\bU|_s+(\bL_s+\bW_s)\bigr)\\
&& \qquad=
L^0_t(|\bU|)+2\int_0^t\I{\bU_s=0}\,d\bL_s+2\int_0^t \I{\bU
_s=0}\,dW_s.
\end{eqnarray*}
Here the last term is zero. This can be seen by using isometry and the fact
that $d\q{U}\approx d\q{W}$.
For the second term use Corollary~\ref{cor2L},
\[
2\int_0^t\I{\bU_s=0}\,d\bL_s=\int_0^t \I{\bU_s=0}\,d\bigl
(L^0_s\bigl(|\bU|+(\bL+\bW)\bigr)-L^0_s\bigl(|\bU|-(\bL
+\bW)\bigr)\bigr).
\]
Rearranging gives that
%
%
%
\begin{equation}\label{eqL0-}
L^0_t(|\bU|)=\int_0^t \I{\bU_s=0}\,dL^0_s\bigl(|\bU|-(\bL
+\bW)\bigr).
\end{equation}
Making the same calculation for
\[
\int_0^t\I{\bU_s=0}\,dL^0_s\bigl(|\bU|-(\bL+\bW)\bigr),
\]
we obtain
%
%
%
\begin{equation}\label{eqL0+}
L^0_t(|\bU|)=\int_0^t \I{\bU_s=0}\,dL^0_s\bigl(|\bU|+(\bL
+\bW)\bigr).
\end{equation}
\eqref{eqL0-} and \eqref{eqL0+} together prove the statement.
\end{pf}

Now, our example is obtained by interlacing the two-dimensional local martingale
$(\bU,\bW)$ from Lemma~\ref{lembZbW} with an independent Brownian motion~$\bB$.
The linkage between the two processes is
\[
\bTV_t=\tfrac{1}{2}\bigl(L^0_{t}(|\bU|+\bL+\bW)+L^0_{t}\bigl
(|\bU|-(\bL+\bW)\bigr)\bigr)
\]
on the one side, and
\[
\bS_t=\max_{s\leq t} \bB_s
\]
on the other side. That is, the processes are time changed so that
after the
time change, $\bTV$ and $\bS$ coincide.
To describe this, put
\begin{eqnarray*}
\alpha(t)&=&
\inf\{u>0\dvtx\bTV_u>\bS_{t-u}\}, \qquad(\TV,L,U,W)_t =(\bTV,\bL
,\bU
,\bW)_{\alpha(t)},\\
\beta(t)&=&\inf\{u>0\dvtx\bS_u>\bTV_{t-u}\}, \qquad(B,S)_t =(\bB
,\bS
)_{\beta(t)}.
\end{eqnarray*}

%
\begin{proposition}\label{prop19}
The following properties hold almost surely:
\begin{longlist}
\item\label{prop19it1} $\alpha$, $\beta$ are nondecreasing, continuous
and
$\alpha(t)+\beta(t)=t$ for all $t\geq0$;
\item\label{prop19it2}
$\lim_{t\to\infty}\alpha(t)=\lim_{t\to\infty}\beta(t)=\infty$;
\item\label{prop19it3} $S=\TV$;
\item\label{prop19it4} for all $t\geq0$, if
$B_t\neq S_t$ then $|L_t+W_t|=|U_t|$.
\end{longlist}
\end{proposition}
\begin{pf} The key property of $\bS$ and $\bTV$ is that they do
not have a
nondegenerate plateau (interval of constancy) at the same level.
The sample path of $\bTV$ is nondecreasing, and therefore $p(\bTV)$
the set
of levels, at which $\bTV$ spends positive amount of time, is at most
countable. The same holds for $\bS$. By the independence of the two
processes, $p(\bTV)$ and $p(\bS)$ are disjoint almost surely.

By the continuity of $\bTV$ and $\bS$, we have
%
%
%
\begin{equation}\label{eqab}
\bTV_{\alpha(t)}=\bS_{t-\alpha(t)}\quad\mbox{and}\quad
\bS_{\beta(t)}=\bTV_{t-\beta(t)}.
\end{equation}
It follows that $\alpha(t)=t- \beta(t)$ almost surely for all $t$. To see
this we can assume on the contrary that $\alpha(t)<t-\beta(t)$. Then
%
\[
\bTV_{\alpha(t)}=\bS_{t-\alpha(t)}\geq\bS_{\beta(t)}=
\bTV_{t-\beta(t)}\geq\bTV_{\alpha(t)},
\]
showing that $\bTV$ and $\bS$ have a nondegenerate plateau at the
same level,
which can happen only on a negligible exceptional event.
Hence $\alpha(t)+\beta(t)=t$ for all $t\geq0$ almost surely.

Since clearly, $\alpha,\beta$ are nondecreasing, the fact that
$\alpha(t)+\beta(t)=t$ implies that they are continuous, even
contractions, that is, $|\alpha(t)-\alpha(s)|\leq|t-s|$
and similarly
for $\beta$. This proves Property (\ref{prop19it1}).\vadjust{\goodbreak}

Property (\ref{prop19it2}) follows from the unboundedness of $\bTV
$ and
$\bS$, cf. (\ref{lembZbWit4}) of Lemma~\ref{lembZbW}.

Property (\ref{prop19it3}) is an easy corollary of \eqref{eqab} and
$\alpha(t)+\beta(t)=t$.

For Property (\ref{prop19it4}) note that if $B_t\neq S_t$
then
$\bB_{\beta(t)}\neq\bS_{\beta(t)}$ and $\bS$ has a nondegenerate
plateau at the
level $S_t$.
But, then $\bTV$ spends zero time at this level, that is, $\alpha(t)$
is a
point of increase of $\bTV$.
Using \eqref{eq2L} this implies that $|\bL+\bW|=|\bU|$
holds at $\alpha(t)$, that is, $|L_t+W_t|=|U_t|$.
\end{pf}

We obtained that $(\alpha(t))\tgo$ and $(\beta(t))\tgo$ are
continuous time
changes with respect to the filtration $(\bF_t)\tgo$ and $(\bG
_t)\tgo$,
respectively, where $\bF_t=\F^{\bU,\bW}_t\vee\sigma(\bB)$ and
$\bG_t=\F^{\bB}_t\vee\sigma(\bU,\bW)$. Then $(U,W)_t=(\bU,\bW
)_{\alpha(t)}$ is
a continuous local martingale in the time changed filtration $\bF
_{\alpha(t)}$,
and since it is clearly adapted to $\F_t=\bF_{\alpha(t)}\cap\bG
_{\beta(t)}$ we
get that $(U,W)$ is a continuous local martingale in $(\F_t)\tgo$. By similar
reasoning, $B_t=\bB_{\beta(t)}$ is also a continuous local martingale
in $(\F_t)\tgo$.

By the definition of $\bTV$ and Corollaries~\ref{cor2L} and \ref
{cortechnical},
we have that
\[
\bL_t=\int_0^t -\sign_0(\bL_s+\bW_s)\,d\bTV_s,
\]
where $\sign_0(x)=\I{x>0}-\I{x<0}$.
Then the same identity holds for the time changed processes, that is,
%
%
%
\begin{equation}\label{eqLTV}
L_t=\int_0^t -\sign_0(L_s+W_s)\,d\TV_s.
\end{equation}

The final step is to define
%
%
%
\begin{equation}\label{eqYVdef}
Y_t=\int_0^t -\sign_0(L_s+W_s)\,dB_s\quad\mbox{and}\quad V=Y+W.
\end{equation}
It is easy to check that $U$ and $V$ are strongly orthogonal, and $U$ is
divergent. Property (\ref{lembZbWit2}) of Lemma~\ref{lembZbW}
is inherited by $U,V$, that is,
$\q{U}_t=\int_0^t Q_s\,d\q{V}_s$ with some
$Q$.
To show that the pair $U,V$ satisfies \eqref{eqUV} we apply the balayage
formula: for a predictable bounded process $\xi$ and a continuous
semimartingale $Z$, we have
\[
\xi_{\gamma(t)} Z_t=\int_0^t \xi_{\gamma(u)}\,dZ_u,
\]
where $\gamma(t)=\sup\{s\leq t\dvtx Z_s=0\}$; see~\cite{MR2200733},
Lemma 0.2, or
\cite{Yor}, Chapter VI. We apply this for $Z=\TV-B$ and
$\xi_t=-\sign_0(L_t+W_t)$, that is,
$\gamma(t)=\sup\{s\leq t\dvtx\TV_s=B_s\}$. Observe that on the interval
$[\gamma(t),t]$ the time change $\alpha$ is constant, hence
$\xi_t=\xi_{\gamma(t)}$ for all $t\geq0$.
Then
\[
L_t-Y_t=\int_0^t \xi_s\,d(\TV_s-B_s)=\xi_t(\TV_t-B_t)=\xi_t(S_t-B_t).
\]
This formula shows that $L_t\neq Y_t$ implies that $S_t\neq B_t$ and hence
$|L_t+W_t|=|U_t|$ by Property (\ref{prop19it4}) of Proposition
\ref{prop19}. That is, if $|V_t|\geq|U_t|$ for some $t$
then either
$Y_t\neq L_t$ and then $|L_t+W_t|=|U_t|$, or $Y_t=L_t$ and we
get that
$|L_t+W_t|\geq|U_t|$. Since $|L+W|\leq|U|$ by the
definition
of $L$, we obtain in both cases that $|L_t+W_t|=|U_t|$. In formula,
\[
\I{|V_t|\geq|U_t|}\leq\I{|L_t+W_t|=|U_t|}
\quad\mbox{and}\quad
\I{|V_t|<|U_t|}\geq\I{|L_t+W_t|<|U_t|}.
\]
Finally, we can write the time-changed version of Proposition~\ref{prop18}
[the time-change $(\alpha(t))\tgo$ is continuous]
\[
dU_t=\I{|L_t+W_t|<|U_t|}\,dU_t=\I{|V_t|<|U_t|}\,dU_t;
\]
that is, \eqref{eqUV} holds.

We can summarize this section in the next theorem.
%
%
\begin{theorem}\label{thm22}
There is a pair $(U,V)$
of strongly orthogonal continuous local martingales such that \eqref{eqUV}
holds, $U,V$ are divergent, and
$d\q{U}$ is absolutely continuous with respect
to $d\q{V}$.
\end{theorem}

For our final statement in this subsection, recall that by \eqref
{eqqNM} when
we reformulate the example in terms of $M$ and $N$, we have
\[
\q{M}=\q{V}=\q{W}+\q{Y}\quad\mbox{and}\quad\q{N}=\q{U}+\q{Y}.
\]
Now, since our construction yields an example in which $d\q{U}$ and
$d\q{W}$ are
equivalent and $\q{U}$, $\q{V}$ are divergent, the same properties
hold for $\q{M}$ and $\q{N}$.
Then, by time change we can transform $(M,N)$ such that $M$ becomes a Brownian
motion and $N$ a continuous local martingale in the time-changed filtration.

%
%
\begin{theorem}\label{thm23}
There is a pair $B,N$ of continuous strongly orthogonal local martingales
such that $B$ is a Brownian motion, $\q{N}_t=\int_0^t Q_s\,ds$ with some
strictly positive $Q$ such that the solution of
\[
dX_t=\sign(X_t)\,dB_t+dN_t
\]
is not pathwise unique.

In other words, if the perturbation of the Tanaka equation is not strong
enough, then pathwise uniqueness of the solution does not hold.
\end{theorem}

The other possibility is that we transform $N$ into a Brownian motion.
Then we
obtain an example showing that in some cases even a Brownian motion is not
strong enough as a perturbation.
%
%
\begin{theorem}\label{thm24}
There is a pair $M,B$ of continuous strongly orthogonal local martingales
such that $B$ is a Brownian motion, $\q{M}_t=\int_0^t Q_s\,ds$ with some
strictly positive $Q$ such that the solution of
\[
dX_t=\sign(X_t)\,dM_t+dB_t
\]
is not pathwise unique.\vadjust{\goodbreak}
\end{theorem}

\subsection{\texorpdfstring{Proof of Lemma \protect\ref{lembZbW}}{Proof of Lemma 17}}\label{sec31}\label{secproofl15}

Lemma~\ref{lembZbW} states the existence of two-dimensional local
martingale
$(U,W)$ 
with essentially
the following property holding almost surely: one
can draw the graph of a continuous function with locally bounded variation
into the plane region
%
%
%
\begin{equation}\label{eqregion}
\{(t,x)\in\real_+\times\real\dvtx-W_t-|U|_t\leq x\leq
-W_t+|U_t|\},
\end{equation}
since this property together with Proposition~\ref{prop16} below ensures
(\ref{lembZbWit3}) of Lem\-ma~\ref{lembZbW}.

To achieve this we start with two independent Brownian motions $\bU$ and~$W$.
Then we apply a time change onto $\bU$ to obtain
$U_t=\bU_{\eta(t)}$.
This time change is in the form
%
%
%
\begin{equation}\label{eqeta}
\eta(t)=\inf\biggl\{s\dvtx\int_0^s |\bU_u|^\kappa du>t\biggr\},
\end{equation}
with a suitably chosen $\kappa>0$.
This way of construction guarantees that (\ref{lembZbWit1}), (\ref
{lembZbWit2}) and even (\ref{lembZbWit4}) of Lemma \ref
{lembZbW} hold.

\renewcommand\thelonglist{\arabic{longlist}}
\renewcommand\labellonglist{(\thelonglist)}
As a result of the time-change the Brownian motion $\bU$
is accelerated when it is near the origin. It has three effects:
\begin{longlist}
\item\label{effect1}
The Hausdorff dimension of the zero level
set $\z(U)$ will decrease below $1/2$.
\item\label{effect2}
Short excursions of $\bU$ after the time change will be
even shorter, and therefore the sum,
which played a crucial role in the proof of Theorem~\ref{thm3},
will be finite, that is,
%
%
%
\begin{equation}\label{eqdW}
\sum_{I\in\cC(U,s)} |\Delta_I W|<\infty\qquad\mbox{almost surely
for all $s\geq0$}.
\end{equation}
\item\label{effect3} To describe the third effect we denote by
$K$ the continuous process
with $K_t=-W_t$ whenever $U_t=0$ and linear in between.

Then, the random closed set $\{t\geq0\dvtx|K_t+W_t|\leq|U_t|\}$
contains in its interior $\z(U)$, the zero level set of $U$, almost
surely. Moreover, if $I$ is a short excursion interval of $U$, then 
$|K+W|\leq|U|$ with high probability.
Then by means of the Borel--Cantelli lemma it follows that
$|K+W|\leq|U|$ on all, but finitely many excursion intervals
ending before $t$, for any $t>0$. That is, the number of exceptional excursion
intervals is locally finite.
\end{longlist}

Properties (\ref{effect1}) and (\ref{effect2}) imply that the
process $K$
defined in (\ref{effect3}) has locally bounded variation. Then property
(\ref{effect3}) implies it is possible to draw a graph of locally bounded
variation into the plain region \eqref{eqregion}: one has to modify
$K$ on
the finitely many exceptional excursion intervals. It is possible since
$|U|\geq\eps$ with some $\eps>0$ on 
the closed set $\cl{A_T}$, where
$A_T=\{t\in[0,T]\dvtx|K_t+W_t|>|U_t|\}$.

So we only have to show that with suitable choice of $\kappa>0$ properties
(\ref{effect1}), (\ref{effect2}) and (\ref{effect3}) are fulfilled.

Property (\ref{effect1}) is a classical fact
(see, e.g.,~\cite{MR0345224}, Section 6.7), where it was proved that
$\dim\z(U)=(2+\kappa)^{-1}$.\vadjust{\goodbreak}

The finiteness of \eqref{eqdW} is a corollary of
\[
\sum_{I\in\cC(U,s)}|I|^{1/2}<\infty\qquad\mbox{for all $s>0$}.
\]
This latter follows from the rather crude estimation on the length of
$I$. If the corresponding excursion interval of $\bU$ is $J$, then
\[
|I|\leq|J| \sup_{s\in J} |\bU_s|^\kappa=
|J|^{1+\kappa/2} \sup_{s\in J} \zfrac{|\bU_s|}{|J|^{1/2}}^\kappa.
\]
Here $\sup_{s\in J}(|\bU_s|/|J|^{1/2})^\kappa$ where
$J$ run
through $\cC(\bU,s)$ is an i.i.d. sequence with finite expectation,
and hence it
is enough to show that
\[
\sum_{J\in\cC(U,s)} |J|^{1/2+\kappa/4}<\infty.
\]
This follows from a trivial modification of Proposition~\ref{lemsum}, as
already mentioned in Remark~\ref{remalpha}.

It remains to show property (\ref{effect3}). In this step
the crucial issue is the estimation of the probability
%
%
%
\begin{equation}\label{eqp}
\mathbf{P}(\exists t\in I_{n}, |K_t+W_t|> |U_t|),
\end{equation}
where $(I_n)_{n\geq1}$ is the usual $\sigma(U)$ measurable
enumeration of the
excursions of $U$.

Let us fix $n$ and drop the index from the notation.
By the definition of $K$ the process $K+W$ is a Brownian bridge on the
interval $I$ and is independent of $U$. Let us map $[0,1]$ onto $I=(a,b)$
linearly by $\phi(t)=t(b-a)+a$ and scale both $K+W$ and $U$ with
$|I|^{-1/2}$. This way we obtain
\begin{eqnarray*}
B_t&=&|I|^{-1/2}\bigl(K_{\phi(t)}+W_{\phi(t)}\bigr) , \\
E_t&=&|I|^{-1/2}|U_{\phi(t)}| .
\end{eqnarray*}
Then $B$ is a standard Brownian bridge, and $E$ is a distorted Brownian
excursion. Now the question is the probability
\[
\mathbf{P}(\exists t\in[0,1], B_t> E_t),
\]
since by symmetry the twice of this probability gives an upper bound for~\eqref{eqp}.
We can describe the graph of the distorted excursion $(E_t)_{t\in
[0,1]}$ in
terms of a standard Brownian excursion $(\bE_t)_{t\in[0,1]}$ and the length
$|J|$ of the excursion interval of $\bU$ which is transformed
after the time
change into $I$. Indeed, the excursion of $\bU$ is obtained by scaling
form $\bE$; that is, its graph can be described as
\[
\{(\bar{a}+|J|t,|J|^{1/2} \bE_t)\dvtx t\in[0,1]\},
\]
where $\bar{a}=\inf J$. To describe the effect of the time-change on the
graph introduce the process
\[
r(t)=\int_0^t |\bE_s|^\kappa ds,\qquad t\in[0,1].
\]
Then $|I|=|J|^{1+\kappa/2} r(1)$, and we can parametrize the
graph of
$E$ as
\[
\biggl\{\biggl(\frac{r(t)}{r(1)},|J|^{-\kappa/4}\frac{\bE
_t}{r(1)^{1/2}}\biggr)\dvtx t\in[0,1]\biggr\}.
\]

Next we define independent variables
\begin{eqnarray*}
\xi&=&\sup_{t\in(0,1)} \frac{B_t}{(t(1-t))^{1/4}},\\
\zeta
&=&\sup_{t\in(0,1)}
\frac{(r(t)(r(1)-r(t)))^{1/4}}{\vphantom{\bar{{\bE}}} \bE_t}=
|J|^{-\kappa/4}\sup_{t\in(0,1)}\frac{(t(1-t))^{1/4}}{E_t} .
\end{eqnarray*}
The point here is that if
$B_{t_0}>E_{t_0}$ for some $t_0\in[0,1]$, then $\xi\zeta|J|^{\kappa/4}>1$.
Whence, by the independence of $\zeta,\xi$ and $|J|$, we have
the next
estimate for the conditional probability,
%
%
%
\begin{equation}\label{eqcondp}
\mathbf{P}(\exists t\in[0,1], B_t>E_t||J|)\leq
\mathbf{P}(\xi\zeta>x)|_{x=|J|^{-\kappa/4}} .
\end{equation}
Hence we are interested in the tail of $\xi$ and $\zeta$.
Although it would be nice to find some explicit formulas, a rather coarse
estimate is sufficient for our purposes. We use that if $B$ is a Brownian
bridge, then $W_t=(1+t)B_{t/(1+t)}$ is a Brownian motion, and
\begin{eqnarray*}
\mathbf{P}(\xi>x)&\leq&2\mathbf{P}\biggl(\sup_{t\in(1/2,1)}\frac
{B_t}{(t(1-t))^{1/4}}>x\biggr)
\\ &\leq&
2 \mathbf{P}\biggl(\exists t\geq0, W_t>\frac12 x(1+t)^{3/4}\biggr).
\end{eqnarray*}
The next lemma shows that the tail of $\xi$ is really thin.
%
%
\begin{lemma}
Let $W$ be a Brownian motion.
Then for $\beta> 1/2$,
\[
\mathbf{P}\bigl(\exists t\geq0, W_t>x(1+t)^\beta\bigr)\leq\frac
{e^{-cx^2}}{1-e^{-cx^2}},
\]
where $c>0$ depends only on $\beta$. For $\beta\in(1/2,1)$ with
$c(\beta)=2\beta
(1-\beta)\* (1/2)^{1/(2\beta-1)}$ the estimate holds.
\end{lemma}
\begin{pf}It is enough to prove for $\beta\in(1/2,1)$.
Take an increasing sequence $(t_n)_{n\geq0}$ such that $t_0=0$ and
$\lim_{n\to\infty}t_n=\infty$. Let $e_k$ denote the secant line
through $t_k,t_{k+1}$, that is,
\[
e_k(t)=\frac{f(t_{k+1})-f(t_k)}{t_{k+1}-t_k} (t-t_k)+f(t_k)=a_kt+b_k ,\vadjust{\goodbreak}
\]
where $f(t)=(1+t)^\beta$. Since $e_k(t)\leq f(t)$ for $t\in
[t_k,t_{k+1}]$ we
have that
\[
\mathbf{P}\bigl(\exists t\geq0, W_t\geq xf(t)\bigr)\leq\sum
_{k=0}^\infty
\mathbf{P}\bigl(\exists t\geq0, W_t\geq xe_k(t)\bigr)=\sum
_{k=0}^\infty
e^{-2x^2a_kb_k}.
\]
In the last step we have used that for the Brownian motion $W$, and $x,y>0$,
we have $\mathbf{P}(\exists t\geq0,W_t\geq x+yt)=e^{-2xy}$;
see, for example, (1) on page 251 of~\cite{MR1912205}.

To finish the proof we need to esimate $a_kb_k$ from below, where
\begin{eqnarray*}
a_k&=&\frac{f(t_{k+1})-f(t_k)}{t_{k+1}-t_k}\geq
f'(t_{k+1})=\beta(1+t_{k+1})^{\beta-1},\\
b_k&=&f(t_k)-t_k\frac{f(t_{k+1})-f(t_k)}{t_{k+1}-t_k} 
\geq
t_k\biggl(\frac{f(t_k)}{t_k}-f'(t_k)\biggr) \\
&=&
t_k
(1+t_k)^{\beta-1}\biggl(1+\frac1{t_k}-\beta\biggr)
\geq
(1-\beta)(1+t_k)^{\beta}.
\end{eqnarray*}
Hence
\[
a_kb_k\geq\beta
(1-\beta) (1+t_k)^\beta(1+t_{k+1})^{\beta-1}\geq
\beta
(1-\beta)\frac{(1+t_{k})^{2\beta}}{1+t_{k+1}}.
\]
Taking $t_k=(k+1)^{1/(2\beta-1)}-1$, we get that $a_kb_k\geq
(k+1)\beta
(1-\beta)(1/2)^{1/(2\beta-1)}$ and
\[
\mathbf{P}\bigl(\exists t\geq0, W_t\geq xf(t)\bigr)\leq\sum
_{k=0}^\infty
e^{-2x^2a_kb_k}\leq
\frac{e^{-cx^2}}{1-e^{-cx^2}}
\]
with $c(\beta)=2\beta
(1-\beta)(1/2)^{1/(2\beta-1)}$.
\end{pf}

%
%
\begin{corollary}
There are $c_1,c_2>0$ such that
\[
\mathbf{P}(\xi>x)\leq c_1e^{-c_2 x^2}\quad\mbox{and}\quad
\mathbf{P}(\bxi>x)\leq c_1e^{-c_2 x^2},
\]
where $\bxi=\sup_{t\in(0,1)} (t(1-t))^{-1/4}\bE_t$.
\end{corollary}

The estimation for the standard Brownian excursion $\bE$ follows from the
description of $\bE$ as a three-dimensional Bessel bridge; that is,
$\rho_{t}=(1+t)\bE_{t/(1+t)}$ is a three-dimensional Bessel process starting
from zero; see~\cite{Yor}, XII, Theorem 4.2. Then, it follows that
$\bE^2\eqinlaw B^2(1)+B^2(2)+B^2(3)$ where $B(1),B(2),B(3)$ are three
independent Brownian bridges. This explains the second part of the corollary.

We will also use the well-known fact about the three-dimensional Bessel
process $\rho$, that
\[
J=\frac{\inf\{\rho_t\dvtx t\geq1\}}{\rho_1}
\]
is independent of $\sigma(\{\rho_s\dvtx s\leq1\})$
and uniformly distributed on $[0,1]$. Formulating this with $\bE$ and
$\bE_{1-t}$ we obtain
that
\[
J_1= \frac{1}{2\bE_{1/2}}\cdot\min_{t\in[\fraca12,1)}\frac{\bE
_{t}}{1-t}
\quad\mbox{and}\quad
J_2= \frac{1}{2\bE_{1/2}}\cdot\min_{t\in(0,\fraca12]}\frac{\bE
_{t}}{t}
\]
are uniformly distributed on $[0,1]$, and $J_1,J_2,\bE_{1/2}$ are independent.

Using these tools we want to estimate
\[
\mathbf{P}(\zeta>x)=\mathbf{P}\biggl(\sup_{t\in(0,1)}\frac
{(r(t)(r(1)-r(t)))^{1/4}}{\bE_t}>x\biggr).
\]

With the notation of the previous corollary,
\[
r(t)\bigl(r(1)-r(t)\bigr)\leq{\bxi}^{2\kappa} \bigl(t\wedge
(1-t)\bigr)^{1+\kappa/4}.
\]
For the denominator we have the following lower bound:
\begin{eqnarray*}
\bE_t & \geq&\biggl(t \cdot\min_{t\in(0,\fraca12]} \frac{\bE
_t}{t}\biggr)\wedge
\biggl((1-t) \cdot\min_{t\in[\fraca12,1)} \frac{\bE_t}{1-t}\biggr
)
\\ & \geq&
2\bigl(t\wedge(1-t)\bigr)\bE_{1/2}(J_1\wedge J_2).
\end{eqnarray*}
Thus for $\kappa\geq12$,
\[
\zeta\leq\frac{\bxi^{\kappa/2}}{2\bE_{1/2}(J_1\wedge J_2)}.
\]

The tail of $\bxi$ and $\xi$ goes to zero exponentially fast, while
on the other
hand
the tail of $(\bE_{1/2}(J_1\wedge J_2))^{-1}$ is polynomial, more precisely,
\begin{eqnarray*}
&&\mathbf{P}\biggl(\frac1{\bE_{1/2}(J_1\wedge J_2)}>x\biggr)\\
&& \qquad\leq\mathbf{P}(\bE_{1/2}<x^{-1/2})+\mathbf{P}(J_1\wedge
J_2<x^{-1/2})\leq
c_3 x^{-1/2}
\end{eqnarray*}
with some positive $c_3$. So we obtain that
%
%
%
\begin{equation}\label{eqhalf+eps}
\mathbf{P}(\xi\zeta>x)\leq c(\eps) x^{-1/2+\eps},
\end{equation}
where $\eps>0$ arbitrary small, and $c(\eps)$ is a positive constant depending
on $\eps$.

Combining \eqref{eqhalf+eps} 
with \eqref{eqcondp} and taking into account Remark \ref
{remalpha}, we get
\[
\mathbf{P}\Bigl(\Bigl\{I\in\cC(U,s)\dvtx\sup_{t\in I}
|K_t+W_t|-|U_t|>0\Bigr\}
\mbox{ is finite}\Bigr)=1\qquad\mbox{for all $s>0$},
\]
that is, the number of excursion intervals of $U$ on which $|K+W|\leq
|U|$ does not hold, is locally finite almost surely, provided that
$\kappa\geq12$.
This proves property (\ref{effect3}) completely.

The next proposition showing the extremal property of $L$ finishes the proof
of Lemma~\ref{lembZbW}.
%
%
\begin{proposition}\label{prop16}
Assume that $f,g,h\dvtx[0,\infty)\to\real$ are continuous
functions, satisfying
$f\leq h\leq g$ and $f(0)=g(0)$.
Then, for any $t\geq0$ the total variation of $L=\bL(f,g)$ on $[0,t]$
is not greater than that of $h$.
\end{proposition}
\begin{pf}
Take $t\geq0$ and a subdivision $t_0=0<t_1<\cdots<t_n=t$. It is
enough to
show that there is a subdivision $s_0=0<s_1<\cdots<s_m=t$ such that
\[
\sum_{j=1}^n |L(t_j)-L(t_{j-1})|\leq
\sum_{j=1}^m |h(s_j)-h(s_{j-1})|.
\]
We may and do assume that the sign of the increments $L(t_j)-L(t_{j-1})$
is alternating on the left. We can simply leave out those $t_j$ at which
the sign of the increments does not alternate without affecting the
left-hand side.

The case $n=1$ and $L(t)=L(0)=0$ is trivial. In all other cases the
increments $L(t_{j})-L(t_{j-1})$, $j=1,\ldots,n$ are nonzero.

If $L(t_{j})-L(t_{j-1})>0$, then there is $s_j\in[t_{j-1},t_{j}]$ such
that
$L(t_j)=f(s_j)\leq h(s_j)$; similarly if $L(t_{j})-L(t_{j-1})<0$, then
there
is $s_j\in[t_{j-1},t_{j}]$ such that $L(t_j)=g(s_j)\geq h(s_j)$.
Defining
$s_0=0$ and $s_{n+1}=t$ we get $|L(t_{j})-L(t_{j-1})|\leq
|h(s_{j})-h(s_{j-1})|$ for $j=1,\ldots,n$ and the statement follows.
\end{pf}

\section*{Acknowledgments}
The author is grateful to Walter Schachermayer for encouraging this
work and
to the anonymous referee whose comments helped to improve the paper.


%

\printaddresses

\end{document}